\documentclass[a4paper,12 pt]{article}
\usepackage[utf8]{inputenc}
\usepackage{typearea}
\usepackage{amssymb,amsmath,amsfonts,amsthm}
\usepackage{times}
\usepackage{mathrsfs}
\usepackage{epic}
\usepackage{times}
\usepackage[english]{babel}

\usepackage[dvipsnames]{xcolor}
\usepackage{hyperref,thm-restate}
\usepackage[capitalize]{cleveref}
\usepackage[mathlines]{lineno}
\usepackage[T1]{fontenc}

\usepackage{graphicx}
\usepackage{subcaption}
\usepackage{float}
\usepackage{accents}

\title{On the extremal values of the cyclic continuants
of Motzkin and Straus }
\author{Alessandro De Luca\textsuperscript{1}, 
Luca Q. Zamboni\textsuperscript{2}}
\date{}

\ExplSyntaxOn
\cs_new_eq:NN \IfSingleItemTF \tl_if_single:nTF
\ExplSyntaxOff

\newcommand\rev[1]{
  \IfSingleItemTF{#1}%
    {#1^*}
    {\left(#1\right)^*}{}} 

\newcommand*{\dt}[1]{%
  \accentset{\mbox{\large\bfseries .}}{#1}}
\newcommand{\nats}{\mathbb N}
\newcommand{\reals}{\mathbb R}
\newcommand{\A}{\mathbb A}
\newcommand{\B}{\mathbb B}
\newcommand{\C}{\mathcal C}
\newcommand{\eps}{\varepsilon}
\newcommand{\Kc}{K^\circlearrowright}
\newcommand{\Ks}{ \dt {K}}
\newcommand{\Ksc}{ \dt {K}^\circlearrowright}

\newcommand{\Sing}{\mathcal {S}}
\newcommand{\Singalt}{\mathcal {S}_{alt}}
\newcommand{\U}{\mathcal {U}}
\newcommand{\Ualt}{\mathcal {U}_{alt}}
\newcommand{\X}{\mathfrak {X}}
\newcommand{\Ac}{{\mathbb A}^\circlearrowright}
\newcommand{\Bc}{{\mathbb B}^\circlearrowright}

\newcommand{\pv}{\mathbf v}
\newtheorem{theorem}{Theorem}[section]
\newtheorem{corollary}[theorem]{Corollary}
\newtheorem{lemma}[theorem]{Lemma}
\newtheorem{proposition}[theorem]{Proposition}

\theoremstyle{definition}
\newtheorem{remark}[theorem]{Remark}
\newtheorem{definition}[theorem]{Definition}
\newtheorem{example}[theorem]{Example}
\newtheorem{alg}[theorem]{Algorithm}

\begin{document}

\maketitle
\begin{center}
\textsuperscript{1}DIETI, Universit\`a di Napoli Federico II,
via Claudio 21, 80125 Napoli, Italy\\
\texttt{alessandro.deluca@unina.it}\\
\textsuperscript{2}  Institut  Camille  Jordan, CNRS  UMR  5208,
Universit\'{e} de Lyon, Universit\'{e} Lyon  1,\\
43 boulevard du 11 novembre 1918, F69622 Villeurbanne Cedex, France\\
\texttt{zamboni@math.univ-lyon1.fr}
\end{center}

\vspace{0.5cm}

\begin{abstract} In a 1983 paper, G. Ramharter asks what are the extremal arrangements for the cyclic analogues of the  regular and semi-regular continuants first introduced by T.S. Motzkin and E.G. Straus in 1956. In this paper we answer this question by showing that for each set $\A$ consisting of positive integers $1<a_1<a_2<\cdots <a_k$  and a $k$-term partition $P: n_1+n_2 + \cdots + n_k=n,$ there exists a unique (up to reversal) cyclic word $x$ which maximizes (resp. minimizes) the regular cyclic continuant $\Kc(\cdot)$ amongst all cyclic words over $\A$ with Parikh vector $(n_1,n_2,\ldots,n_k).$ We also show that the same is true for the minimizing arrangement for the semi-regular cyclic continuant $\Ksc(\cdot).$
As in the non-cyclic case, the main difficulty is to find the maximizing arrangement for $\Ksc(\cdot)$, which is not unique in general and may depend on the integers $a_1,\ldots,a_k$ and not just on their relative order. We show that if a cyclic word $x$ maximizes $\Ksc(\cdot)$  amongst all permutations of $x,$ then it verifies a strong combinatorial condition which we call the singular property. We develop an algorithm for constructing all singular cyclic words having a prescribed Parikh vector.

\smallskip\noindent
\textbf{Keywords:} Cyclic continuants, Regular and semi-regular continued
fractions, Christoffel words, Singular words.
\end{abstract}


\section{Introduction}
Given a set $\A$ consisting of positive integers $a_1<a_2<\cdots <a_k$  and a $k$-term partition $P: n_1+n_2 + \cdots + n_k=n,$  find the extremal denominators of the regular and semi-regular continued fraction $[0;x_1,x_2,\ldots,x_n]$ with partial quotients $x_i\in \A$ and where each $a_i$ occurs precisely $n_i$ times in the sequence $x_1,x_2,\ldots,x_n.$ The regular continuant $K_n(x_1,x_2,\ldots,x_n)$ is defined recursively by $K_0()=1$, $K_1(x_1)=x_1$ and  \begin{equation}\label{K1} K_n(x_1,x_2,\ldots,x_n)=x_nK_{n-1}(x_1,x_2,\ldots ,x_{n-1}) + K_{n-2}(x_1,x_2,\ldots ,x_{n-2})\end{equation} and is equal to the denominator of the finite regular  continued fraction $[0; x_1,x_2,\ldots ,x_n]$. Similarly, the semi-regular continuant 
 $\Ks_n(x_1,x_2,\ldots,x_n)$
is defined recursively by $\Ks_0()=1$, $\Ks_1(x_1)=x_1$ and 
\begin{equation}\label{K2} \Ks_n(x_1,x_2,\ldots,x_n)=x_n\Ks_{n-1}(x_1,x_2,\ldots ,x_{n-1}) - \Ks_{n-2}(x_1,x_2,\ldots ,x_{n-2}).\end{equation} 
For semi-regular continuants, the digit $1$ needs to be excluded and in this case, $\Ks(x)$ is the denominator of the terminating semi-regular 
continued fraction
\[ [x]^{ \bullet}=\cfrac{1}{x_1 -\cfrac{1}{x_2 -\cfrac{1}{
\ddots-\cfrac{1}{x_n}}}} \]

In the framework of the regular continuant $K_n(\cdot)$, T.S. Motzkin and E.G. Straus~\cite{MS} provided a first partial answer in the special case in which each $n_i=1.$ In \cite{Cu2}, T.W. Cusick found the maximizing arrangement when $\A=\{1,2\}.$ But the general problem was settled by G. Ramharter \cite{Ram}.
Ramharter  gave an explicit description of both extremal arrangements for the regular continuant as well as the minimizing arrangement for the semi-regular continuant (see Theorem 1 in \cite{Ram}). Ramharter also gave a description of the maximizing arrangement for the semi-regular continuant in the binary case (i.e., when $k=2)$ and showed in this case that the maximizing arrangement  is a finite Sturmian word, and furthermore that the palindromic binary maximizing arrangements for $\Ks(\cdot)$ are in one-to-one correspondence with the extremal cases of the Fine and Wilf theorem \cite{FW} with two co-prime periods (see Theorem~3 in \cite{Ram05}). In all four cases mentioned above, Ramharter showed that the extremal arrangements are unique (up to reversal) and independent of the actual values of the digits $a_i.$ However, the determination of the maximizing arrangement for $\Ks(\cdot)$ for general $k\geq 3$ turned out to be more difficult. He conjectured that for general $k$ the maximizing arrangement for $\Ks(\cdot)$ is always unique (up to reversal) and depends only on the $k$-term partition $P,$ i.e., not on the actual values of the $a_i.$ This conjecture was verified by the authors in the case of a ternary alphabet $\A$ (i.e., $k=3)$ but was also shown to fail for general $k\geq 4$ (see \cite{DeLEdZam}). 

In \cite{Ram}, Ramharter asks what are the extremal arrangements for the cyclic analogues of $K$ and $\Ks$ in the sense of  \cite{MS} (see also \cite{Ram}) defined by: 
 \[\Kc(x_1x_2\cdots x_n)=K(x_1x_2\cdots x_n)+ K(x_2\cdots x_{n-1})\]
  \[\Ksc(x_1x_2\cdots x_n)=\Ks(x_1x_2\cdots x_n)- \Ks(x_2\cdots x_{n-1}).\]
  It is easily verified that if two sequences $x,y\in \A^+$ are cyclic conjugates of one another, i.e., $x=uv$ and $y=vu$ for some sequences $u$ and $v,$ then $\Kc(x)=\Kc(y)$ and $\Ksc(x)=\Ksc(y).$
In other words, the functions $\Kc$ and $\Ksc$ are well defined on cyclic words.  
 In this paper we give an explicit description of the extremal arrangements for  $\Kc$  and the minimizing arrangement for $\Ksc.$ In each case, the arrangement is unique (up to reversal) and only depends on the $k$-term partition $P.$ Thus, as in the linear (i.e., non-cyclic) version, the main difficulty is to understand the maximizing arrangement in the case of the semi-regular continuant.
 We show that if a cyclic word $x$ maximizes $\Ksc(\cdot)$ amongst all permutations of $x,$ then $x$ verifies the following combinatorial criterion: For all factorizations $x=uv$ with $u\neq\rev u$ and $v\neq \rev v$ one has $u\prec\rev u$ if and only if $v\prec\rev v$, where $\prec$ denotes the lexicographic order on $\A^+$ induced by the usual order on $\reals$, and $\rev u$ the mirror image of $u$. We say that a cyclic word $x$ is {\it singular} if it verifies this combinatorial condition. 
  
 In the binary case, we show that associated with every vector $(n_1,n_2)$ is a unique cyclic singular word $x$ with Parikh vector  $(n_1,n_2),$ and in fact we show that $x$ is a power of a Christoffel word $C_{p.q}$ with relatively prime periods $p$ and $q.$ Thus in particular the maximizing arrangement for the cyclic semi-regular continuant is unique in the binary case.  But already on ternary alphabets, there may be several Abelian equivalent singular words.  
 
In general, given an ordered set $\A=\{a_1<a_2<\cdots <a_k\}$  and a $k$-term partition $P: n_1+n_2 + \cdots + n_k=n,$a $k$-term partition $P,$ we wish to describe all cyclic singular words on $\A$ with Parikh vector  $(n_1,n_2,\ldots,n_k).$ 
To this end we develop a general algorithm for constructing all cyclic singular words on an ordered alphabet $\A$ having a prescribed Parikh vector. The algorithm  consists of a non-commutative variant of the usual Euclidean algorithm coupled with a combinatorial algorithm which constructs the singular words from the output of the arithmetic part. This algorithm also allows us to construct all linear singular words in the sense of \cite{DeLEdZam} and hence can be used to find the maximizing arrangements for the (linear) semi-regular continuant.  
In the cyclic binary case, such an arrangement is unique and in fact given by the
Christoffel word having the prescribed Parikh vector.
In contrast, for
larger values of $k$,
the maximizing arrangement for $\Ksc(\cdot)$ need not be unique and may actually depend on the values $a_1,a_2,\ldots, a_k.$

\section{Preliminaries}

Let $\A$ be a finite non-empty and let $\A^+$ denote the free semigroup generated by $\A$ consisting of all finite words $x=x_1x_2\cdots x_n$ with each $x_i\in \A.$ We set $\A^*=\A^+\cup \{\varepsilon\}$ where $\varepsilon$ denotes the empty word. We let $\Ac$ denote the set of all cyclic words over $\A,$ i.e., 
$\Ac = \A^+ / \!\sim$ where $\sim$ is the equivalence relation on $\A^+$ defined by $w\sim w'$ if and only if $w=uv$ and $w'=vu$ for some $u,v\in \A^*.$
In order to avoid confusion, we will refer to the elements of $\A^+$ as {\it linear words}. Thus, each cyclic word $\omega$ is represented by one or more linear words $x.$ At times, by abuse of notation, we write $\omega=uv$ with $u,v\in \A^+$ when in fact we mean that $\omega$ admits a linear representation  which factors as $uv.$ For $x=x_1x_2\cdots x_n\in \A^+$ we put $x^*=x_nx_{n-1}\cdots x_1.$ Similarly, for $\omega\in \Ac$ we let $\omega^*\in \Ac$ denote the class of $x^*$ where $x$ is any linear representation of $\omega.$ 
For $\omega\in \A^+\cup \Ac$ and $a\in \A,$ let $|w|_a$ denote the number of occurrences of the letter $a$ in $w.$
More generally, for $u\in\A^+$, we let $|w|_u$ denote the number of occurrences of
$u$ in $w$ as a factor.
Two cyclic words $\omega$ and $\nu $ are said to be Abelian equivalent if $|\omega|_a=|\nu|_a$ for each $a\in \A.$ We let ${\mathcal X}(\omega)$ denote the {\it cyclic Abelian class} of $\omega$ consisting of all cyclic words which are Abelian equivalent to $\omega,$ i.e., the set of all permutations of $\omega.$ Thus ${\mathcal X}(\omega)$ is uniquely determined by  $(|\omega|_a)_{a\in \A}$ which is called the Parikh vector of $\omega.$  We let $\mathfrak{X}(\omega)$ denote the {\it symmetric cyclic Abelian class} of $\omega$ in which we identify each cyclic word $\nu$ with its reversal $\nu^*.$  Thus $\mathfrak{X}(\omega)$ is also completely determined by the Parikh vector of $\omega.$ For $\omega\in \Ac$ we let $\bar \omega$ denote the class of $\omega$ in $\mathfrak{X}(\omega),$ i.e., $\bar \omega =\{\omega,\omega^*\}.$

Given a linear word  $x=x_1x_2\cdots x_n\in \A^+,$  let $K(x)$ denote the regular continuant polynomial $K_n(x_1,x_2,\ldots,x_n)$ defined recursively by $K_0()=1,$ $K_1(x_1)=x_1$ and \[K_n(x_1,x_2,\ldots,x_n)=x_nK_{n-1}(x_1,x_2,\ldots ,x_{n-1}) + K_{n-2}(x_1,x_2,\ldots ,x_{n-2}).\]
We regard $K(x)$ as a formal polynomial in $x_1,x_2,\ldots ,x_n.$ If the $x_i$ are positive integers, then $K(x)$ is the denominator of the terminating regular 
continued fraction $[x]=[0; x_1,x_2,\ldots ,x_n].$ Similarly, we define the semi-regular continuant polynomial $\Ks(x)=\Ks_n(x_1,x_2,\ldots,x_n)$
defined recursively by $\Ks_0()=1,$ $\Ks_1(x_1)=x_1$ and \[\Ks_n(x_1,x_2,\ldots,x_n)=x_n\Ks_{n-1}(x_1,x_2,\ldots ,x_{n-1}) - \Ks_{n-2}(x_1,x_2,\ldots ,x_{n-2}).\]
If the $x_i$ are positive integers each strictly greater than $1,$  then $\Ks(x)$ is the denominator of the terminating semi-regular 
continued fraction \[ [x]^{ \bullet}=\cfrac{1}{x_1 -\cfrac{1}{x_2 -\cfrac{1}{
      \begin{array}{@{}c@{}c@{}c@{}}
        x_3 - {}\\ &\ddots\\ &&{}- \cfrac{1}{x_n}
      \end{array}
    }}} \]
    We recall a few key facts concerning the regular and semi-regular continuants (see for example \cite{Perron}): For each $x=x_1x_2\cdots x_n\in \A^+$ one has $K(x)=K(x^*)$ and $\Ks(x)=\Ks(x^*).$  Also, for each $m\in \{1,\ldots ,n-1\}$ one has 
    \begin{equation}\label{Preg}K(x)=K(x_1\cdots x_m)K(x_{m+1}\cdots x_n) + K(x_1\cdots x_{m-1})K(x_{m+2}\cdots x_n)\end{equation} while 
      \begin{equation}\label{Psing}\Ks(x)=\Ks(x_1\cdots x_m)\Ks(x_{m+1}\cdots x_n) - \Ks(x_1\cdots x_{m-1})\Ks(x_{m+2}\cdots x_n)\end{equation}
      where we adopt the convention that $K(x_j\cdots x_{j-1})=\Ks(x_j\cdots x_{j-1})=1.$

      Finally, if each $x_i$ is a positive integer strictly greater than $1,$ then one has 
     \begin{equation}\label{CF} [x]=\frac{K(x_2\cdots x_n)}{K(x_1x_2\cdots x_n)}\,\,\,\,\,\, \mbox{and} \,\,\,\,\,\,\, [x]^{ \bullet}=\frac{\Ks(x_2\cdots x_n)}{\Ks(x_1x_2\cdots x_n)}.\end{equation}
     
 \section{Extremal arrangements}

 We shall be interested in the cyclic analogues of $K$ and $\Ks$ in the sense of  \cite{MS} (see also \cite{Ram}) defined by: 
 \[\Kc(x_1x_2\cdots x_n)=K(x_1x_2\cdots x_n)+ K(x_2\cdots x_{n-1})\]
  \[\Ksc(x_1x_2\cdots x_n)=\Ks(x_1x_2\cdots x_n)- \Ks(x_2\cdots x_{n-1}).\]
    
It is easily verified that if $x,y\in \A^+$ are cyclic conjugates of one another, i.e., $x\sim y,$ the $\Kc(x)=\Kc(y)$ and $\Ksc(x)=\Ksc(y).$
In other words, the functions $\Kc$ and $\Ksc$ are well defined on cyclic words.  
The function $\Kc$ has also been considered in \cite{BeBoCa}, under the name {\it circular continuant}, in connection with Hopcroft's automaton minimization algorithm.

The following proposition and its proof should be compared with Theorem~3 in \cite{Ram}:

\begin{proposition}\label{val} Let $\A\subseteq \{2,3,4,\ldots \}$ and let $<$ (resp. $<_{alt})$ denote the usual (resp. alternating) lexicographic order on $\A^+$ induced by the usual order on $\reals.$ Let $\omega \in \Ac.$  Assume $\omega=uv$ with $u\neq u^*$ and $v\neq v^*,$ and let $\omega '$ be the cyclic word represented by $u^*v.$
\begin{enumerate}
\item If either $u<_{alt} u^*$ and $v^*<_{alt} v$ or $u^*<_{alt} u$ and $v<_{alt} v^*$ then $\Kc(\omega ')<\Kc(\omega ).$
\item   If either $u< u^*$ and $v^*< v$ or $u^*< u$ and $v< v^*$ then $\Ksc(\omega ')>\Ksc(\omega ).$ 
\end{enumerate}  
\end{proposition}  

\begin{proof} Let us write $u=u_1u_2\cdots u_m,$  $v=v_1v_2\cdots v_n$ and $x= u_1u_2\cdots u_m v_1v_2\cdots v_n.$
By application of (\ref{Preg}) and (\ref{CF}) we have
\begin{align*}K(uv)&=K(u)K(v) + K(u_1\cdots u_{m-1})K(v_2\cdots v_n)\\
&=K(u)K(v) + K(u_{m-1}\cdots u_1)K(v_2\cdots v_n)\\
&=K(u)K(v) + [u^*]K(u^*)[v]K(v)\\
&= K(u)K(v) + [u^*][v]K(u)K(v).
\end{align*}
Replacing $u$ by $u^*$ gives  \[K(u^*v)=K(u)K(v) +[u][v]K(u)K(v).\]
Thus
\[K(u^*v)-K(uv)= ([u]-[u^*])[v]K(u)K(v).\]
Similarly
\begin{align*}K(u_2\cdots u_mv_1\cdots v_{n-1}) &= K(u_2\cdots u_m)K(v_1\cdots v_{n-1})+ K(u_2\cdots u_{m-1})K(v_2\cdots v_{n-1})\\
&= K(u_2\cdots u_m)K(v_{n-1}\cdots v_{1})+ K(u_{m-1}\cdots u_2)K(v_2\cdots v_{n-1})\\
&=[u]K(u)[v^*]K(v) + K(u_{m-1}\cdots u_2)K(v_2\cdots v_{n-1})
\end{align*}
while 
\[K(u_{m-1}\cdots u_1v_1\cdots v_{n-1}) = [u^*]K(u)[v^*]K(v)+ K(u_{m-1}\cdots u_2)K(v_2\cdots v_{n-1}) \]
so taking the difference gives
\[K(u_{m-1}\cdots u_1v_1\cdots v_{n-1}) - K(u_2\cdots u_mv_1\cdots v_{n-1}) = ([u^*]-[u])[v^*]K(u)K(v).\]
Finally
\begin{align*}\Kc(\omega ')-\Kc(\omega )&=([u]-[u^*])[v]K(u)K(v) +  ([u^*]-[u])[v^*]K(u)K(v)\\
&=([u]-[u^*])([v]-[v^*])K(u)K(v).
\end{align*}
But  the assumption that $u<_{alt} u^*$ and $v^*<_{alt} v$ or $u^*<_{alt} u$ and $v<_{alt} v^*$ implies that $\mbox{sgn}([u]-[u^*])=-\mbox{sgn}([v]-[v^*])$ and therefore $\Kc(\omega ')-\Kc(\omega )<0$ as required. 

As for 2. we have
\[\Ks(uv)=\Ks(u)\Ks(v)-[u^*]^\bullet [v]^\bullet\Ks(u)\Ks(v)\]
while
\[\Ks(u^*v)=\Ks(u)\Ks(v)-[u]^\bullet[v]^\bullet\Ks(u)\Ks(v).\]
Similarly
\[\Ks(u_2\cdots u_mv_1\cdots v_{n-1})=[u]^\bullet[v^*]^\bullet\Ks(u)\Ks(v) -\Ks(u_2\cdots u_{m-1})\Ks(v_2\cdots v_{n-1})\]
while
\[\Ks(u_{m-1}\cdots u_1v_1\cdots v_{n-1})=[u^*]^\bullet[v^*]^\bullet\Ks(u)\Ks(v)-\Ks(u_{m-1}\cdots u_2)\Ks(v_2\cdots v_{n-1}).\]
Therefore
\begin{align*}
\Ksc(\omega ')-\Ksc(\omega )&= \Ks(u^*v)-\Ks(u_{m-1}\cdots u_1v_1\cdots v_{n-1}) \\&- (\Ks(uv) -\Ks(u_2\cdots u_mv_1\cdots v_{n-1}))\\
&=\Ks(u^*v) - \Ks(uv) \\&+ \Ks(u_2\cdots u_mv_1\cdots v_{n-1})-\Ks(u_{m-1}\cdots u_1v_1\cdots v_{n-1})\\
&=([u^*]^\bullet - [u]^\bullet)[v]^\bullet \Ks(u)\Ks(v) + ([u]^\bullet-[u^*]^\bullet)[v^*]^\bullet\Ks(u)\Ks(v)\\
&=([u^*]^\bullet - [u]^\bullet)([v]^\bullet-[v^*]^\bullet)\Ks(u)\Ks(v).
\end{align*}
The assumption $u< u^*$ and $v^*< v$ or $u^*< u$ and $v< v^*$ implies that $\mbox{sgn}([u]^\bullet-[u^*]^\bullet)=-\mbox{sgn}([v]^\bullet-[v^*]^\bullet)$ and therefore $\Ksc(\omega ')-\Ksc(\omega )>0$ as required. \end{proof}

Henceforth, let $\A$ be an ordered finite set. Given $w\in \A^+\cup \Ac$ we let $\max(w)$ (resp. $\min(w))$ denote the largest (resp. smallest) letter $j\in \A$ occurring in $w.$ Let $\prec$ (resp. $\prec_{alt})$ denote the usual (resp. alternating) lexicographic order on $\A^+$ induced by the order on $\A.$ 
In defining $\prec$ we adopt the convention that $u\prec v$ whenever $v$ is a proper prefix of $u$ which is of course opposite to the true dictionary order. 
Similarly, if $v$ is a proper prefix of $u$ and $|v|$ is even, then $u\prec_{alt} v,$ while if $|v|$ is odd then $v\prec_{alt}u.$
For $u,v \in \A^+$ we write $u\preceq v$ (resp. $u\preceq_{alt} v)$ to mean either $u=v$ or $u\prec v$ (resp. $u\prec_{alt} v.)$ 

Let $\omega$ be a cyclic word over the alphabet $\A.$ A factorization $\omega=uv$ with $u,v\in \A^+$ and with $u\neq u^*$ and $v\neq v^*$ is said to be {\it synchronizing} if the inequality between $u$ and $u^*$ is the same as the inequality between $v$ and $v^*,$ i.e., $u\prec u^*$ if and only if $v\prec v^*.$ Otherwise the factorization is  said to be non-synchronizing.  For example, one can check that for the cyclic word $\omega =aaabaab$ over the ordered alphabet  $\{a< b\},$ all factorizations of $\omega$  as the product of two non-palindromes are synchronizing. On the other hand the cyclic word $\omega '=aaaabab$ contains both synchronizing factorizations, for example  $(aaaab)(ab),$ as well as non-synchronizing factorizations, for example $(aab)(abaa).$
We are only interested in factorizations of the form $\omega=uv$ in which both $u$ and $v$ are non-palindromes. For this reason, henceforth whenever we write $\omega=uv$ it shall be implicitly assumed that $u\neq u^*$ and $v\neq v^*.$  Typically a cyclic word will admit both synchronizing and non-synchronizing factorizations. We note that if a factorization $\omega=uv$ is synchronizing, then so is the factorization $\omega^*=u^*v^*.$ Similarly we say a factorization $\omega=uv$ is {\it alt-synchronizing} if $u\prec_{alt} u^*$ if and only if $v\prec_{alt} v^*.$ 

Set 
\[\Sing(\A) =\{\omega \in \Ac : \, \mbox{all factorizations $\omega=uv$ are synchronizing}\}\]
\[\Singalt(\A)=\{\omega \in \Ac : \, \mbox{all factorizations $\omega=uv$ are alt-synchronizing}\}\]
\[\U(\A)=\{\omega \in \Ac : \, \mbox{no factorization $\omega=uv$ is synchronizing}\}\]
\[\Ualt(\A)=\{\omega \in \Ac : \, \mbox{no factorization $\omega=uv$ is alt-synchronizing}\}\]
Given a cyclic Abelian class $\C$ over the ordered alphabet $\A,$ we put $\Sing(\C)=\C\cap \Sing(\A)$ and if $\X$ denotes the symmetric cyclic Abelian class corresponding to $\C,$ then we put $\Sing(\X)=\{\bar {\omega}: \omega \in \Sing(\C)\}.$ The sets $\Singalt(\C), \Singalt(\X), \U(\C), \U(\X), \Ualt(\C)$ and $\Ualt(\X)$ are defined analogously.

\noindent As an immediate consequence of Proposition~\ref{val} we have:

\begin{proposition}\label{maxmin} Let $\A\subseteq \{2,3,4,\ldots \}$ where we regard $\A$ to be ordered according to the usual order $2<3<4<\cdots .$  Let $\C$ be a cyclic Abelian class over $\A$ and let $\omega \in \C.$  
\begin{enumerate}
\item If  $\Kc(\omega) =\min\{\Kc(\nu): \nu \in \C\},$ then $\omega\in \Singalt (\C).$ 
\item  If  $\Kc(\omega) =\max\{\Kc(\nu): \nu \in \C\},$ then $\omega\in \Ualt(\C)$ 
\item If  $\Ksc(\omega) =\min\{\Ksc(\nu): \nu \in \C\},$ then $\omega\in \U (\C).$ 
\item If  $\Ksc(\omega) =\max\{\Ksc(\nu): \nu \in \C\},$ then $\omega\in \Sing(\C).$ 
\end{enumerate}
\end{proposition} 

In particular, if $\A$ is any ordered alphabet and $\X$ any symmetric cyclic Abelian class over $\A,$ the sets $\Sing(\X),$ $\Singalt(\X),$
$\U(\X)$ and $\Ualt(\X)$ are each non-empty. We will next show that $|\Singalt(\X)|=|\Ualt (\X)|=|\U(\X)|=1.$ Thus in particular, if $\A\subseteq \{2,3,4,\ldots,\},$ then each cyclic Abelian class $\C$ over $\A$ contains a unique (up to reversal) cyclic word $\omega$ which minimizes $\Kc,$ i.e., $\Kc(\omega) =\min\{\Kc( \nu): \nu \in \C\}.$ Moreover, up to word isomorphism, $\omega$ is independent of the actual choice of $\A.$

\begin{lemma}\label{minreg00}Let $\omega \in \Singalt(\A).$ Put $a=\min(\omega)$ and $d=\max(\omega).$ Then either every occurrence of $a$ in $\omega$ is followed (resp. preceded) by $d$ or every occurrence of $d$ in $\omega$ is followed (resp. preceded) by $a.$
\end{lemma}

\begin{proof}We first note that if $ab$ and $dc$ are factors of $\omega$ with $b,c\in \A,$ then either $b=d$ or $c=a.$ In fact, we may write $\omega =uv$ where $u$ begins in $b$ and ends in $d$ and $v$ begins in $c$ and ends in $a.$ Thus, if $b\neq d$ then $u\prec_{alt} u^*$ and hence $v\preceq_{alt}v^*$ which implies that $c=a.$  It follows from this that either every occurrence of $a$ in $\omega$ is an occurrence of $ad$ or every occurrence of $d$ in $\omega$ is an occurrence of $da.$ Similarly, if $ba$ and $cd$ are factors of $\omega$ with $b,c\in \A,$ then either $b=d$ or $c=a.$
Thus either every occurrence of $a$ in $\omega$ is preceded by $d$ or every occurrence of $d$ in $\omega$ is preceded by $a.$ \end{proof}

\begin{lemma}\label{minreg0} Let $\omega \in \Singalt(\A).$ Put $a=\min(\omega)$ and $d=\max(\omega).$ Then $\omega$ or $\omega^*$ admits a linear representation of the form $(da)^Nz$ for some $N\geq 1,$ $z\in \A^*$ such that:
\begin{enumerate}
\item[i)] $da$ is not a factor of $z.$
\item [ii)] If $z=va$ with $v\in \A^*$  then $v^*\preceq_{alt} v.$
\item [iii)] If $z=dv$  with $v\in \A^*$  then $v^*\preceq_{alt} v.$
\item[iv)] $|z|_d|z|_a\leq 1$ and if $|z|_d\geq 1$ then $z$ begins in $d$ while if $|z|_a\geq 1,$ then $z$ ends in $a.$
\end{enumerate}
\end{lemma}

\begin{proof} 
By Lemma~\ref{minreg00} either $ad$ is a factor of $\omega$ or $da$ is a factor of  $\omega.$ Thus $da$ is a factor of $\omega$ or of $\omega^*.$ 
We next claim that if $da$ is a factor of $\omega$ (resp. $\omega^*)$, then all the occurrences of $da$ in $\omega$ (resp. $\omega^*)$ are clustered together. In fact, suppose $\omega$ (resp. $\omega^*)$ contains a factor of the form $rsdaxr's'da$ with $x\in \A^*.$ We will show that either $rs=da$ or $r's'=da.$ We begin by showing that either $s=a$ or $s'=a.$ We can write $\omega=uv$ with $u=axr's'$ and where $v$ begins in $da$ and ends in $sd.$ If $s'\neq a,$ then $u\prec_{alt} u^*$ whence $v\preceq_{alt} v^*$ which implies $s=a.$ Thus without loss of generality, we may assume that $s'=a.$ Then by Lemma~\ref{minreg00}, either $r'=d$ or $s=a.$ If  $r'=d$ then we are done, thus we may suppose that $r'\neq d$ whence $s=a.$  So $\omega$ (resp. $\omega^*$) contains the factor $radaxr'ada.$ So we may write $\omega=uv$ with $u=daxr'$  and where $v$ begins in $ad$ and ends in $ra.$ 
Since $u^*\prec_{alt} u$ we have $v^* \preceq_{alt} v$ which implies $r=d$ as required.  Similarly, all occurrences of $ad$ in $\omega$ are contiguous. 

It follows that we may represent  $\omega$ (or $\omega^*$ or both) by a linear word of the form $(da)^Nz$ where $da$ is not a factor of $z.$ Let us pick such a representative with $N$ maximal. Without loss of generality we may suppose that $\omega =(da)^Nz$ with $da$ not a factor of $z$ and if $\omega^*=(da)^Mz'$ with $da$ not a factor of $z',$ then $M\leq N.$ 

Let us now turn our attention to conditions ii) and iii). Note that $z$ cannot both begin in $d$ and end in $a$ otherwise we would have $\omega^*=(da)^{N+1}z'$ contradicting the maximality assumption on the exponent $N.$ Thus conditions  ii) and iii) cannot both fail. If both conditions ii) and iii) are verified, then we are done. If condition ii) fails, then we replace $z$ by $z'=v^*a.$ Note that since $a(da)^N$ is a palindrome, replacing the prefix $v$ of $z$ by $v^*$ amounts to replacing $\omega$ by $\omega^*.$ In particular, condition i) is still verified. Note also that $z'$ does not begin in $d$ (since $da$ does not occur in $z)$ whence conditions i),  ii) and  iii) are now verified. Similarly, if condition iii) fails, then we replace $z$ by $z'=dv^*.$ Since $(da)^Nd$ is a palindrome, replacing the suffix $v$ of $z$ by $v^*$ amounts to replacing $\omega$ by $\omega^*.$ Note $z'$ does not end in $a$ (since $da$ does not occur in $z)$ and so conditions i), ii) and  iii) are now verified.  

With regards to condition iv), we will show that $|z|_d\geq 1 \Longrightarrow |z|_a\leq 1$ and  $|z|_a\geq 1 \Longrightarrow |z|_d\leq 1.$
So assume first that  $|z|_d\geq 1.$  We will show that $|z|_a\leq 1.$ We first claim $z$ begins in $d.$  If not,  then we have an occurrence of $a$ in $\omega$ not followed by $d$ whence by Lemma~\ref{minreg00} every occurrence of $d$ in $\omega$ must be followed by $a.$ This implies that $z$ must contain an occurrence of $da$ which contradicts condition i). Thus $z$ must begin in $d.$ Now suppose to the contrary that $|z|_a>1.$ Since the prefix $d$ of $z$ is not followed by $a$ every occurrence of $a$ in $\omega$ must be followed by $d.$ In particular, since $|z|_a>1$ there must be an occurrence of $ad$ in $z.$ So we can write $az=adxady$ with $x,y \in \A^*$ which implies that either $x=\varepsilon$ or $x$ is a power of $ad.$ Either way this contradicts i).  

Similarly, assume $|z|_a\geq 1$ and we will show that $|z|_d\leq 1.$ If $z$ does not end in $a$ then we have an occurrence of $d$ in $\omega$ which is not preceded by $a$ and hence every occurrence of $a$ in $\omega$ must be preceded by $d.$ So this produces an occurrence of $da$ in $z$ contradicting i).
So $z$ must end in $a$ and this $a$ is not preceded by $d.$ Whence every occurrence of $d$ in $\omega$ must be preceded by $a.$ Now assume to the contrary that $|z|_d>1.$ This produces an occurrence of $ad$ in $z.$ So as before, this leads to an occurrence of $da$ in $z$ which contradicts i). 
Our proof also shows that if $|z|_d\geq 1$ then $z$ begins in $d$ and if $|z|_a\geq 1$ then $z$ ends in $a.$ \end{proof}

\begin{lemma}\label{minreg1} Let $\omega \in \Singalt(\A).$ Put $a=\min(\omega)$ and $d=\max(\omega).$ Then $\omega$ or $\omega^*$ admits a linear representation of the form $dx$ where $x$ verifies the following prefix/suffix condition: For every proper prefix (resp. proper suffix) $u$ of $x$ one has $u\preceq_{alt}u^*$ (resp. $u^*\preceq_{alt} u).$
\end{lemma}

\begin{proof} By Lemma~\ref{minreg0},  short of  exchanging $\omega$ and $\omega^*,$ we may write $\omega= (da)^Nz$ ($N\geq 1)$ where  $z$ verifies conditions i) through iv) of Lemma~\ref{minreg0}.  Thus we may write $\omega = (da)^nz(da)^m$ with $n\geq 1$ and $m\in \{n,n-1\}.$
Set $x=a(da)^{n-1}z(da)^m.$ We now show that $x$ verifies the prefix/suffix condition stated in the lemma. 

Let $u$ be a proper prefix of $x.$ If $u$ is a prefix of $a(da)^{n-1}$ then clearly $u\preceq_{alt} u^*.$  Next suppose $2n-1<|u|<2n-1 + |z|.$ 
Suppose to the contrary that $u^*\prec_{alt} u.$ If $n\geq 2,$ then $u$ begins in $ad$ and ends in $da.$ This puts an occurrence of $da$ in $z,$ a contradiction. If $n=1,$ then since $u$ begins in $a$ it follows that $u$ ends in $a$ and hence  $|z|_a\geq 1.$ By condition iv) of Lemma~\ref{minreg0} it follows that $z$ ends in $a$ and since $u$ is a proper prefix of $az$ we deduce that $|z|_a\geq 2,$ whence $|z|_d=0.$ So we can write $\omega=duv(da)^m$ with $m=0,1$ and where $v$ is a non-empty proper suffix of $z.$  Since $u^*\prec_{alt} u,$ it follows that $(v(da)^md)^*\preceq_{alt} v(da)^md$ which is a contradiction since $|v|_d=0.$ 
Next suppose $u=a(da)^{n-1}z.$ If $n>1,$ since $da$ does not occur in $z$ we must have $u\preceq_{alt}u^*.$ If $n=1,$ then $u=az.$ To see that $u\preceq_{alt} u^*,$ suppose to the contrary that $u^*\prec_{alt}u.$ Then $z^*a\prec az$ which implies that $z=va$ for some $v\in \A^*$ and $u^*=av^*a\prec_{alt} ava=u$ or equivalently $v\prec_{alt}v^*$ which contradicts condition ii) of Lemma~\ref{minreg0}.  
Finally suppose that $|u|>2n-1+|z|.$ Then the only case which could potentially be a problem is if $m=n$ and $u=a(da)^{n-1}z(da)^{n-1}$ and $z$ ends in $a.$ In this case, by condition ii) of Lemma~\ref{minreg0} we may write $z=va$ with $v^*\preceq_{alt}v.$ Since $|a(da)^{n-1}|$ is odd it follows that $a(da)^{n-1}va(da)^{n-1}\preceq_{alt} a(da)^{n-1}v^*a(da)^{n-1},$ in other words $u\preceq_{alt} u^*.$

Now let $u$ be a proper suffix of $x.$ If $u$ is a suffix of $(da)^m$ then clearly $u^*\preceq_{alt} u.$  Next suppose $2m<|u|<2m+|z|.$ First, if  $|u|=2m+1$ then either $u=a(da)^m,$ in which case $u=u^*$ or $u^*\prec_{alt} u.$ So suppose $2m+1<|u|<2m+|z|$ and suppose to the contrary that $u\prec_{alt}u^*.$  
If $m\geq 1,$ then $u^*$ begins in $ad$ and hence so does $u.$ But this implies that $ad$ is a factor of $z$ which in particular, by condition iv), implies that $z$ ends in $a$ and hence $|z|_a\geq 2.$ Together with $|z|_d\geq 1,$ this contradicts condition iv). If $m=0$ then $n=1$ and we can write $\omega = davu$ where $v$ is a non-empty proper prefix of $z.$ As $u\prec_{alt}u^*,$ it follows that $dav\preceq_{alt} (dav)^*.$  If $|v|\geq 2,$ then $v$ ends in $ad$ which again contradicts condition iv). If $|v|=1,$ then $\omega = dadu$ and $z=du$ with  $u\prec_{alt}u^*.$ This contradicts condition  iii) of Lemma~\ref{minreg0}. Next suppose $u=z(da)^m.$ Since we already considered the case in which $|u|=2m+1$ we may assume that $|z|\geq 2.$ 
First suppose $m\geq 1;$  if $u\prec_{alt}u^*$ then $u$ begins in $ad$ which produces an occurrence of $ad$ in $z$ and hence contradicts condition iv) of  
Lemma~\ref{minreg0}. If $m=0$ then $n=1$ and $\omega =dau.$ Since $(da)^*\prec_{alt}da$ it follows that $u^*\preceq_{alt} u.$ Finally suppose that $|u|>2m+|z|.$ Then the only case which could potentially be a problem is if $m=n-1$ and $u=a(da)^{n-2}z(da)^{n-1}$ and $z$ begins in $d.$ Writing $z=dv$ and applying condition iii) and using the fact that $|a(da)^{n-2}d|$ is even gives $v^*\preceq_{alt}v \Longrightarrow a(da)^{n-2}dv^*(da)^{n-1} \preceq_{alt} a(da)^{n-2}dv(da)^{n-1}  \Longrightarrow u^*\preceq_{alt}u$ as required. This completes the proof of the lemma.  \end{proof}

\begin{lemma}\label{minreg2} Let $\omega  \in \Singalt(\A). $ Write $\omega=dx$ where $x$ satisfies the prefix/suffix condition of Lemma~\ref{minreg1}.
Then for every factorization $x=u^*vw$ with $v\neq v^*$ and $u\neq w$ one has $v\prec_{alt}v^*$ if and only if $w \prec_{alt} u.$
\end{lemma}

\begin{proof} Fix a factorization $x=u^*vw$ with $v\neq v^*$ and $u\neq w.$ Assume $v\prec_{alt}v^*.$ Then since  $\omega  \in \Singalt(\A)$ we have that
$wdu^* \preceq_{alt} udw^*.$ Thus if neither $u$ nor $w$ is a proper prefix of the other, then $w\prec_{alt}u$ as required. This inequality is also verified in case $w$ is a proper prefix of $u$ of odd length or if $u$ is a proper prefix of $w$ of even length. Thus suppose $w$ is a proper prefix of $u$ of even length.
Then as $w^*vw$ is a proper suffix of $x,$ we have that $w^*v^*w\preceq_{alt} w^*vw$ which in turn implies $v^*\preceq_{alt}v$ contrary to our assumption. Similarly if $u$ is a proper prefix of $w$ of odd length, then as $u^*vu$ is a proper prefix of $x,$ we deduce that $u^*vu\preceq_{alt}u^*v^*u$ and hence $v^*u\preceq_{alt} vu$ or equivalently $v^*\preceq_{alt}v,$ again a contradiction to our assumption that $v\prec_{alt}v^*.$

Now assume that $v^*\prec_{alt} v.$ We will show that $u\prec_{alt}w.$ Since  $\omega  \in \Singalt(\A)$ we have that
$   udw^* \preceq_{alt} wdu^*.$ Thus if neither $u$ nor $w$ is a proper prefix of the other, then $u\prec_{alt}w$ as required. This inequality is also true in case $u$ is a proper prefix of $w$ of odd length or if $w$ is a proper prefix of $u$ of even length. So suppose $u$ is a proper prefix of $w$ of even length. Then $u^*vu$ is a proper prefix of $x$ and hence $u^*vw\preceq_{alt}u^*v^*u$ which implies that $v\preceq v^*$ contrary to our assumption. Similarly, if $w$ is a proper prefix of $u$ of odd length then $w^*v^*w\preceq_{alt}w^*vw$ which implies that $v\preceq_{alt} v^*,$ a contradiction. \end{proof}

\begin{proposition}For each symmetric cyclic Abelian class $\X$ over the ordered alphabet $\A$ we have $|\Singalt (\X)|=1.$
In particular, if $\A\subseteq \{2,3,4,\ldots,\}$, then  each cyclic Abelian class $\C$ over $\A$ contains a unique (up to reversal) cyclic word $\omega$ with the property that $\Kc(\omega)=\min \{\Kc(\nu) : \nu \in \C\}.$
\end{proposition}

\begin{proof} By Proposition~\ref{maxmin} we have that  $|\Singalt (\X)|\geq 1.$  To see that $|\Singalt (\X)|\leq 1,$ assume that $\bar \omega\in \Singalt(\X).$ By lemmas~\ref{minreg0}, \ref{minreg1} and \ref{minreg2} we have that $\omega$ or $\omega^*$ admits a linear representation of the form $dx$ with $d=\max(\omega)$ and for all factorizations $x=u^*vw$ ($u,v,w \in \A^*)$ with $v\neq v^*$ and $u\neq w$ we have $v\prec_{alt} v^*$ if and only if $w\prec _{alt}u.$
Thus if $ \Singalt(\X)$ contains another cyclic class $\bar \nu \neq \bar \omega,$  then the Abelian class of $x$ would contain another linear word $y$ with the property that   for all factorizations $y=u^*vw$ with $v\neq v^*$ and $u\neq w$ we have $v\prec_{alt} v^*$ if and only if $w\prec_{alt} u.$  But this contradicts Theorem~2 in \cite{Ram}. The last statement of the proposition follows immediately from Proposition~\ref{maxmin}.\end{proof}

\begin{lemma}\label{maxreg} Let $\omega \in \Ualt(\A)$ and let $j=\max(\omega).$ Then either $\omega$ or $\omega^*$ admits a linear representation of the form $y=jxj^{n-1}$ with $n\geq 1,$ $|x|_j=0$ and $x\preceq_{alt} x^*.$ Moreover, for all factorizations $y=u^*vw$ ($u,v,w\in \A^*)$ with $v\neq v^*$ and $u\neq w$ we have $v\prec _{alt}v^*$ if and only if $u\prec_{alt}w.$ 
\end{lemma}

\begin{proof} We begin by showing that $\omega$ admits a linear representation of the form $j^nx$ $(n\geq 1)$ with $|x|_j=0.$ In other words, all occurrences of $j$ in $\omega$ are clustered together. In fact, if this were not the case, then we could write $\omega=jxajyb$ with $a,b\in \A\setminus \{j\}$ and $x,y\in \A^*.$ But then the factorization $\omega=uv$ with $u=jxa$ and $v=jyb$ is alt-synchronizing since $u^*\prec_{alt} u$ and $v^*\prec_{alt} v.$ Having established the claim, short of replacing $\omega$ by $\omega^*,$ we may assume without loss of generality that $\omega=j^nx$ with $x \preceq x^*.$ Put $y=jxj^{n-1}.$  Then $y$ is a linear representation of $\omega.$ 

Now assume $y=u^*vw$ with $v\neq v^*$ and $u\neq w. $ First suppose $v\prec_{alt} v^*.$ We will show that $u\prec_{alt}w.$ Since $\omega \in \Ualt(\A)$ we have that $uw\preceq_{alt} wu^*.$ Thus if $u$ and $w$ are not proper prefixes of one another, then $u\prec_{alt}w.$ So suppose $u$ is a proper prefix of $w.$ We note that $u\neq \varepsilon$ otherwise $v$ would be a non-palindromic prefix of $y$ which would imply that $v^*\prec_{alt}v$ contrary to our assumption that $v\prec_{alt} v^*.$ Thus we can write $u=zj$ for some $z\in \A^*.$ It follows that $w=zj^{n-1}$ with $n>2$ and $x=z^*vz.$
Since $v\prec_{alt} v^*$ and $x\preceq_{alt} x^*$ it follows that $|z|$ is even whence $u\prec_{alt}w.$ Next suppose that $w$ is a proper prefix of $u.$ If $w=\varepsilon$ then $u\prec_{alt}w.$ So suppose $w\neq \varepsilon.$ Then $n=1$ and $w$ is a suffix of $x.$ Let us write $u=wu'j$ with $u'\in \A^*.$ 
If $u'\neq \varepsilon,$ then we can write $y=(ju')(w^*vw)$ and $(ju')^*\prec_{alt}ju'.$ It follows that $w^*vw \preceq_{alt} w^*v^*w$ and hence that $|w|$ is even. Thus, $u\prec_{alt} w.$ Finally, if $u'=\varepsilon,$ then $u=wj$ and $x=w^*vw.$ Since $v\prec_{alt}v^*$ and $x\preceq_{alt} x^*$ it follows that $|w|$ is even and hence $u\prec_{alt}w$ as required.

Next suppose $v^*\prec_{alt}v.$ We will show that $w\prec_{alt}u.$ Since $\omega \in \Ualt(\A),$ we have that $wu^*\preceq_{alt}uw^*.$ Thus is neither $u$ nor $w$ is a proper prefix of the other, then $w\prec_{alt}u.$ So suppose $u$ is a proper prefix of $w.$ If $u=\varepsilon,$ then $w\prec_{alt}u$ as required. If $u\neq \varepsilon,$ then $u=zj$ (with $z\in \A^*),$  $w=zj^{n-1}$ with $n>2$ and $x=z^*vz.$ Since $x\preceq_{alt}x^*$ and $v^*\prec_{alt}v$ it follows that $|z|$ is odd, whence $|u|$ is even and so $w\prec_{alt} u.$ So assume now that $w$ is a proper prefix of $u.$ We note that $w\neq \varepsilon$ otherwise $v$ is a non-palindromic suffix of $y$ so we can write $y=u^*v$ with $u\prec_{alt}u^*.$ But then $v\preceq_{alt}v^*$ contrary to our assumption that  $v^*\prec_{alt}v.$ So, $n=1,$ $u=wj$ and $x=w^*vw.$ Since $x\preceq_{alt}x^*$ and $v^*\prec_{alt}v,$ it follows that $|w|$ is odd and hence $w\prec_{alt}u$ as required.  \end{proof}

\begin{proposition} For each symmetric cyclic Abelian class $\X$ over the ordered alphabet $\A$ we have $|\Ualt (\X)|=1.$
In particular, if $\A\subseteq \{2,3,4,\ldots,\}$, then  each cyclic Abelian class $\C$ over $\A$ contains a unique (up to reversal) cyclic word $\omega$ with the property that $\Kc(\omega)=\max \{\Kc(\nu) : \nu \in \C\}.$
\end{proposition}

\begin{proof} That $|\Ualt (\X)|\geq 1$ follows from Proposition~\ref{maxmin}. In order to show that $|\Ualt (\X)|\leq 1,$ assume $\bar \omega \in \Ualt (\X).$ Then, by Lemma~\ref{maxreg}, the pair $\{\omega,\omega^*\}$ determines a linear word $y$ which is a linear representation of $\omega$ or of $\omega^*,$ and with the property that for all factorizations $y=u^*vw$ ($u,v,w \in \A^*)$ with $v\neq v^*$ and $u\neq w$ we have $v\prec_{alt} v^*$ if and only if $u\prec_{alt} w.$ So if $\Ualt( \X)$ contains another cyclic class $\bar \nu \neq \bar \omega,$  then the Abelian class of $y$ would contain another linear word $z$ with the property that  
for all factorizations $z=u^*vw$ with $v\neq v^*$ and $u\neq w$ we have $v\prec_{alt} v^*$ if and only if $u\prec_{alt} w.$  But this contradicts Theorem~2 in \cite{Ram}. The last statement of the proposition follows immediately from Proposition~\ref{maxmin}. \end{proof}

\begin{lemma}\label{minsemireg} Let $\omega \in \U(\A)$ and let $j=\max(\omega).$ Then either $\omega$ or $\omega^*$ admits a linear representation of the form $y=j^nxj^m$ with $n\geq 1$ and $m\in \{n,n-1\},$ such that $|x|_j=0$ and $x\preceq x^*.$ Moreover, for all factorizations $y=u^*vw$ ($u,v,w\in \A^*)$ with $v\neq v^*$ and $u\neq w$ we have $v\prec v^*$ if and only if $u\prec w.$ 
\end{lemma}

\begin{proof} We begin by showing that $\omega$ admits a linear representation of the form $j^Nx$ $(N\geq 1)$ with $|x|_j=0.$ In other words, all occurrences of $j$ in $\omega$ are clustered together. In fact, if this were not the case, then we could write $\omega=jxajyb$ with $a,b\in \A\setminus \{j\}$ and $x,y\in \A^*.$ But then the factorization $\omega=uv$ with $u=jxa$ and $v=jyb$ is synchronizing since $u^*\prec u$ and $v^*\prec v.$ Having established the claim, short of replacing $\omega$ by $\omega^*,$ we may assume without loss of generality that $\omega=j^Nx$ with $x \preceq x^*.$ Put $y=j^nxj^m$ where $n+m=N$ and  $m\in \{n,n-1\}.$ Then $y$ is a linear representation of $\omega.$ 

Now assume $y=u^*vw$ with $v\neq v^*$ and $u\neq w. $ First suppose $v\prec v^*.$ We will show that $u\prec w.$ By considering the factorization $w=v(wu^*)$ we have that $uw^*\preceq wu^*.$ If neither $u$ nor $w$ is a proper prefix of the other, then $u\prec w$ as required. Also, if $w$ is a proper prefix of $u,$ then $u\prec w.$ So assume that $u$ is a proper prefix of $w.$ We will actually show that this cannot happen. First note that $u\neq \varepsilon$ 
since otherwise $v$ would be a non-palindromic proper prefix of $y$ and hence $v*\prec v$ contrary to our assumption. 
It follows that $u=j^r$ and $w= j^s$ for some $r<s\leq m\leq n.$ Then $v=j^{n-r}xj^{m-s}$ and since $n-r>m-s$ we get that $v^*\prec v,$ a contradiction. 

Next suppose  $v^*\prec v.$  Then $wu^*\preceq uw^*.$ So if neither $u$ nor $w$ is a proper prefix of the other, then $w\prec u.$ Also, if $u$ is a proper prefix of $w$ then $w\prec u$ as required. So let's suppose that $w$ is a proper prefix of $u.$ We will actually show that this cannot happen. First note that $w\neq \varepsilon$ for otherwise $v$ would be a non-palindromic proper suffix of $y$ and hence $v\prec v^*$ contrary to our assumption. 
If $w=j^r$ for some $r,$ then $u=j^s$ for some $s>r.$
But then $v=j^{n-s}xj^{m-r}$ and furthermore $n-s\leq m-r.$ So if $n-s<m-r,$ then $v\prec v^*$ while if $n-s=m-r$ then $v<\prec v^*$ since $x\preceq x^*$ and $v\neq v^*.$ So either way we get a contradiction. Finally, if $w\neq j^r,$ then $w=zj^m$ and $u=zj^{m+1}$ and $x=z^*vz.$ Since $x\preceq x^*$ and $v\neq v^*$ it follows that $v\prec v^*,$ a contradiction. This completes the proof of the lemma.  \end{proof}

\begin{proposition} For each symmetric cyclic Abelian class $\X$ over the ordered alphabet $\A$ we have $|\U (\X)|=1.$
In particular, if $\A\subseteq \{2,3,4,\ldots,\}$, then  each cyclic Abelian class $\C$ over $\A$ contains a unique (up to reversal) cyclic word $\omega$ with the property that $\Ksc(\omega)=\min \{\Ksc(\nu) : \nu \in \C\}.$
\end{proposition}

\begin{proof} That $|\U (\X)|\geq 1$ follows from Proposition~\ref{maxmin}. In order to show that $|\U (\X)|\leq 1,$ assume $\bar \omega \in \U (\X).$ Then, by Lemma~\ref{minsemireg}, the pair $\{\omega,\omega^*\}$ determines a linear word $y$ which is a linear representation of $\omega$ or of $\omega^*,$ and with the property that for all factorizations $y=u^*vw$ ($u,v,w \in \A^*)$ with $v\neq v^*$ and $u\neq w$ we have $v\prec v^*$ if and only if $u\prec w.$ So if $ \U(\X)$ contains another cyclic class $\bar \nu \neq \bar \omega,$  then the Abelian class of $y$ would contain another linear word $z$ with the property that  
for all factorizations $z=u^*vw$ with $v\neq v^*$ and $u\neq w$ we have $v\prec v^*$ if and only if $u\prec w.$  But this contradicts Theorem~2 in \cite{Ram}. The last statement of the proposition follows immediately from Proposition~\ref{maxmin}. \end{proof}

\begin{remark}Let $\A \subseteq \{2,3,4,\ldots\}$ and let $\C$ be a a cyclic Abelian class over $\A.$ For $\omega\in \C, $  if $\omega$ minimizes the valuation $\Ksc$ restricted to $\C,$ then $\omega$ admits a linear representation $y$ which minimizes $\Ks$ restricted to the Abelian class of $y.$ 
\end{remark}

\begin{proposition}\label{lintocirc} Let $j=\max \A.$ Assume $x\in (\A\setminus\{j\})^+.$ Then $xj\in \Sing(\A)$ (resp.  $xj\in \Singalt (\A))$ if and only if for all factorizations $x=u^*vw$ with $v\neq v^*$ and $u\neq w$  one has $v\prec v^* \Longleftrightarrow w\prec u$ (resp.  $v\prec_{alt} v^* \Longleftrightarrow w\prec_{alt} u)$
\end{proposition}

\begin{proof} Assume $xj\in \Sing(\A)$ and that $x=u^*vw$ with $v\neq v^*$ and $u\neq w.$  Then $v\prec v^* \Longleftrightarrow wju^*\preceq ujw^*.$
If $u$ or $w$ are not proper prefixes of one another, then the second inequality is equivalent to $w\prec u.$  Moreover,  if $u$ is a proper prefix of $w$ then the inequalities $w\prec u$ and $wju^*\preceq ujw^*$ are each true while if $w$ is a proper prefix of $u$ then both inequalities are false. 
Next suppose that for all factorisations $x=u^*vw$ with $v\neq v^*$ and $u\neq w$  one has $v\prec v^* \Longleftrightarrow w\prec u.$ Let $\omega \in \Ac$ be represented by $xj.$ We will show that $\omega \in \Sing(\A).$ Write $\omega =UV$ with $U\neq U^*$ and $V\neq V^*.$ Assume without loss of generality that $j$ occurs in $U.$ Set $U=wju^*$ and $V=v.$ Then $x=u^*vw$ with $v\neq v^*$ and $u\neq w.$ Thus $v\prec v^* \Longleftrightarrow w\prec u.$ If $u$ and $w$ are not prefixes of one another then $V\prec V^*  \Longleftrightarrow v\prec v^* \Longleftrightarrow w\prec u \Longleftrightarrow U\prec U^*.$ Also, the latter two inequalities are both true if $u$ is a proper prefix of $w$ and are both false if $w$ is a proper prefix of $u.$

Next assume $xj\in \Singalt(\A)$ and that $x=u^*vw$ with $v\neq v^*$ and $u\neq w.$  Then $v\prec_{alt} v^* \Longleftrightarrow wju^*\preceq_{alt} ujw^*.$ If $u$ or $w$ are not proper prefixes of one another, then the second inequality is equivalent to $w\prec_{alt} u.$ Assume $u$ is a proper prefix of $w;$ if $|u|$ even, then the inequalities $w\prec_{alt} u$ and $wju^*\preceq_{alt} ujw^*$ are each true while if $|u|$ is odd then both inequalities are false. Similarly, assume$w$ is a proper prefix of $u;$ if $|w|$ is odd then the inequalities $w\prec_{alt} u$ and $wju^*\preceq_{alt} ujw^*$ are each true while if $|w|$ is even then both inequalities are false. Conversely suppose that for all factorisations $x=u^*vw$ with $v\neq v^*$ and $u\neq w$  one has $v\prec_{alt} v^* \Longleftrightarrow w\prec_{alt} u.$ Let $\omega \in \Ac$ be represented by $xj.$ We will show that $\omega \in \Singalt(\A).$ Write $\omega =UV$ with $U\neq U^*$ and $V\neq V^*.$ Assume without loss of generality that $j$ occurs in $U.$ Set $U=wju^*$ and $V=v.$ Then $x=u^*vw$ with $v\neq v^*$ and $u\neq w.$  Thus $v\prec_{alt} v^* \Longleftrightarrow w\prec_{alt} u.$ If $u$ and $w$ are not prefixes of one another then $V\prec_{alt} V^*  \Longleftrightarrow v\prec_{alt} v^* \Longleftrightarrow w\prec_{alt} u \Longleftrightarrow U\prec_{alt} U^*.$ Also, the latter two inequalities are both true if $u$ is a proper prefix of $w$  and $|u|$ is even or if $w$ is a proper prefix of $u$ and $|w|$ is odd. Similarly, the latter two inequalities are both false if $u$ is a proper prefix of $w$  and $|u|$ is odd or if $w$ is a proper prefix of $u$ and $|w|$ is even. 
\end{proof}

\section{Cyclic singular words}

 
\begin{definition}\rm{A cyclic word $w\in \Ac$ is said to be {\it singular} (resp. {\it alt-singular}) if all factorisations $w=uv$  are synchronizing (resp. alt-synchronizing). }
\end{definition}

\begin{remark}
    \label{t:1sqlet}
    By definition, if $a,a',c,c'\in\A$ are such that $ac$ and $a'c'$ are
    factors of some cyclic singular word $\omega$, then
    $a<c'$ implies $a'\leq c$.
\end{remark}

We note that the property of being singular (resp. alt-singular) is invariant under reversal, i.e., $w$ is singular (resp. alt-singular) if and only if $w^*$ is singular
 (resp. alt-singular). Thus Proposition~\ref{maxmin} could have been stated in terms of the symmetric cyclic Abelian class $\mathfrak{X}(w).$  We next show that for each cyclic word $w,$ the symmetric cyclic Abelian class $\mathfrak{X}(w)$ contains both singular elements and alt-singular elements:

\begin{proposition}\label{DG} Each symmetric cyclic Abelian class $\mathfrak{X}$ contains both singular elements and alt-singular elements.
\end{proposition}  

\begin{proof} Let $\mathfrak{X}$ be a symmetric cyclic Abelian class with Parikh vector $(n_i)_{i\in \A}.$  The order $\prec$ allows us to endow $\mathfrak{X}$ with the structure of a directed graph as follows: We put a directed edge from the cyclic word $w=uv\in \mathfrak{X}$  to the cyclic word $w'=u^*v\in \mathfrak{X}$ if the factorization $w=uv$ is non-synchronizing. Thus to each non-synchronizing factorization of $w$ corresponds an outgoing edge from $w$ to some other cyclic word. In particular, a cyclic word $w\in \mathfrak{X}$ is singular if and only if $w$ has no outward directed edges. Notice that this construction depends only the Parikh vector corresponding to $\mathfrak{X}$ and the order on the underlying alphabet. In other words, if $\phi: \A\rightarrow \B$ is an order preserving bijection between two finite ordered alphabets $\A$ and $\B$ and $\mathfrak{X}=\mathfrak{X}(w) $ and $\mathfrak{X'}=\mathfrak{X}(w')$ with $w\in \Ac$ and $w'\in \Bc,$ and $w$ and $w'$ determine the same Parikh vector in the sense that $|w|_a=|w'|_{\phi(a)}$ for each $a\in \A,$ then the resulting directed graphs are isomorphic. In other words, if in  $\mathfrak{X}$ there is a directed edge from $x$ to $x'$ then in   $\mathfrak{X'}$ there is a directed edge from $\phi(x)$ to $\phi(x').$ 

In order to show that $\mathfrak{X}$ contains a singular element, it suffices to show that the directed graph  $\mathfrak{X}$ defined above is acyclic. Pick  $\B\subset \{2,3,4,\ldots \}$ and an order preserving bijection $\phi : \A \rightarrow \B.$ Fix $w\in \mathfrak{X}$ and put $\mathfrak{X'}=\mathfrak{X}(\phi(w)).$  It follows from Proposition~\ref{val} that if there is a directed edge from $x$ to $x'$ in $\mathfrak{X'},$ then $\Ksc(x')>\Ksc(x).$ 
It follows from this that $\mathfrak{X'}$ is acyclic and hence the same is true of $\mathfrak{X}.$

To prove the existence of an alt-singular element in  $\mathfrak{X}=\mathfrak{X}(w)$ we proceed analogously with $\prec$ replaced by the alternating order $\prec_{alt}.$ More precisely, we endow  $\mathfrak{X}$ with the structure of a directed graph by putting a directed edge from a cyclic word $x=uv\in \mathfrak{X}$  to the cyclic word $x'=u^*v\in \mathfrak{X}$ if the factorization $x=uv$ is not alt-synchronizing. We then pick set $\B\subset \{2,3,4,\ldots \}$ and an order preserving bijection $\phi : \A \rightarrow \B$ and put  $\mathfrak{X'}=\mathfrak{X}(\phi(w)).$ By applying item 1. of Proposition~\ref{val} with the valuation $\Kc$ we deduce that $\mathfrak{X'}$ is acyclic and hence so is $\mathfrak{X}.$ \end{proof} 

We will later see that both directed graphs defined on $\mathfrak{X}$ in the proof of Proposition~\ref{DG} are connected as graphs. In fact, in each case there is a single vertex having no incoming edges. In other words, there exists a unique $w\in \mathfrak{X}$ with the property that all factorisations $w=uv$ are non synchronizing (resp. non alt-synchronizing). In contrast, a symmetric cyclic Abelian class may contain multiple singular elements. For instance, over the ordered alphabet $\A=\{a<b<c\},$ let $\mathfrak{X}$ denote  the symmetric cyclic Abelian class whose Parikh vector is $(2,2,2).$ Then the  directed graph associated with $\prec$ admits a unique vertex with no incoming edges represented by the cyclic word $aabccb$ which is a palindrome.  In contrast, there are two vertices with no outgoing edges represented by the cyclic words $abcabc$ and $abbcac$ and their reversals.   
 The following proposition shows that on a binary ordered alphabet, each cyclic Abelian class contains a unique singular cyclic word:
 
 \begin{proposition} Let $\A=\{a<b\}$ and let $w \in \Ac.$ The following are equivalent
 \begin{enumerate}
 \item $w$ is singular.
 \item $w$ is balanced.
 \item $w$ is a Christoffel word.
 \end{enumerate}
 \end{proposition}
 
 \begin{proof} Recall that a linear or cyclic word $w$ is said to be {\it  balanced} if for all factors $u$ and $v$ of $w$ of equal length one has $||u|_a-|v|_a|\leq 1.$  Some authors require that $|w|_a$ and $|w|_b$ be coprime in the definition of a Christoffel word. Here we adopt the definition given in \cite{BerdeL} which does not require this condition. The equivalence between 2. and 3. is well known (see for instance \cite{deLDeL}) so we shall only prove that 1. and 2. are equivalent. 
 Assume $w\in \Ac$ is not singular. Pick a factorization $w=uv$ with $u\prec u^*$ and $v^*\prec v.$ Writing $u=xau'bx^*$ and $v=ybv'ay^*$ with
 $x,y,u',v' \in \A^*,$ we see that $w$ contains the factors $bx^*yb$ and $ay^*xa$ whence $w$ is not balanced. Conversely, if $w$ is not balanced then there exists a palindrome $x\in \A^*$ such that both $axa$ and $bxb$ are factors of $w.$ Note also that these two factors cannot overlap one another since each prefix (resp. suffix) of $axa$ contains one additional occurrence of $a$ than the corresponding suffix (resp. prefix) of $bxb$ of the same length. Thus we can write $w=uv$ where $u$ begins in $a$ and ends in $b$ and $v$ begins in $xb$ and ends in $ax.$ It follows that $u\prec u^*$ while $v^*\prec v$ and hence $w$ is not singular.   \end{proof} 
 
 In particular, the last result shows that in the binary case, the maximizing
 arrangement for $\Ksc$ is unique. The following example shows that this is not
 necessarily the case for larger integer alphabets.
 \begin{example}
 Let $\A=\{a<b<c<d<e\}$ and $\pv=(1,2,2,2,1)$. Then it is not difficult to check
 that the cyclic singular words with Parikh vector $\pv$ are:
 \[x=bccdbdae,\quad y=bdbccdae,\quad\text{and } z=bcdbcdae,\]
 along with their reverses. Letting $(a,b,c,d,e)=(2,3,4,5,6)$, we obtain
 $\Ksc(x)=22735$, $\Ksc(y)=22751$, and $\Ksc(z)=22646$, so that the maximizing
 arrangement is $y$. Substituting $(a,b,c,d,e)=(2,3,4,10,11)$ instead,
 $\Ksc(x)=213920$ is the maximum, with $\Ksc(y)=213916$ and $\Ksc(z)=211336$.
 Finally, for $(a,b,c,d,e)=(2,3,4,9,10)$, the maximizing arrangement is not even
 unique, as $\Ksc(x)=\Ksc(y)=153347>151598=\Ksc(z)$.
 \end{example}

The following arithmetic observation was recently used by M.~Lapointe~\cite{Lap19}
to determine the number of orbits of a discrete symmetric interval exchange
transformation. We give here a simple proof for the sake of completeness.
\begin{lemma}
    \label{t:mid}
    Let $\pv=(n_d)_{d\in\A}\in\nats^{\A}$ be a non-zero vector.
    Setting
    \begin{equation}
	    \label{e:deltaj}
    	\delta_b=\delta_b(\pv)=\sum_{\substack{c\in\A\\c>b}}n_c
        -\sum_{\substack{a\in\A\\a<b}}n_a
    \end{equation}
    for all $b\in\A$, one has either
    \begin{enumerate}
        \item there exists a single letter $b$ such that $n_b>|\delta_b|$ and
        $n_{d}<|\delta_{d}|$ for all other letters $d\neq b$, or
        \item there exist two letters $a<c$ such that $n_a=|\delta_a|>0$ and
        $n_c=|\delta_c|>0$; moreover, $n_b=|\delta_b|=0$ if $a<b<c$, and
        $n_{d}<|\delta_{d}|$ if $d<a$ or $d>c$.
    \end{enumerate}
\end{lemma}
\begin{proof}
    Let $N=\sum_{d\in\A}n_d$ and consider the interval $[0,N]$ partitioned into
    sections of length $(n_d)_{d\in\A}$ in order. By straightforward
    calculation, $n_b\geq |\delta_b|$ is equivalent to
    \[\frac N2\geq\max\left\{\sum_{a<b}n_a,\,
    \sum_{c>b}n_c\right\},\]
    that is, the midpoint of the interval falling within the $b$ section.
    This happens to a single sub-interval of positive length if the inequality
    is strict. Otherwise, the midpoint hits the division between two
    sections of positive lengths $n_a$ and $n_c$, possibly separated by others
    of length $0$.
\end{proof}

For $\omega\in\Ac$ with Parikh vector $\pv$, and $b\in\A$, we set
$\delta_b(\omega)=\delta_b(\pv)$, defined as in \cref{e:deltaj}.
We simply write $\delta_b$ for $\delta_b(\omega)$ when the context is clear.
By \Cref{t:1sqlet}, if $\omega$ is singular, then there is at most one letter
$b$ such that $bb$ occurs in $\omega$, and by \Cref{t:mid} there is at most
one $b'$ such that $0<|\delta_{b'}|<|\omega|_{b'}$. As we shall see, such
letters coincide in general.

The following lemma was proved in~\cite{DeLEdZam} in the ternary linear case:

\begin{lemma}
	\label{t:runs}
    Let $\omega\in\Ac$ be a cyclic singular word, and $b\in\A$. The following
    hold:
    \begin{enumerate}
        \item\label{nodd}
        If $\delta_b>0$ (resp.~$\delta_b<0$) and $d,d'\leq b$
        (resp.~$d,d'\geq b$) are letters with $dd'\neq bb$, then $dd'$ does not
        occur in $\omega$.
        \item\label{runs} If $0<|\delta_b|\leq |\omega|_b$, then $\omega$
        contains exactly $|\delta_b|$ runs of consecutive $b$, and no factor
        $dd'$ where $d,d'\in\A$ are both larger or both smaller than $b$.
    \end{enumerate}
\end{lemma}
\begin{proof}
    Without loss of generality, suppose $\delta_b>0$, the case $\delta_b<0$
    being completely symmetric. Assume $|\omega|_b<|\delta_b|$ first. This is
    equivalent to
    \[|\omega|_b+\sum_{a<b}|\omega|_a<\sum_{c>b}|\omega|_c\]
    so that $aa'\in L(\omega)$ for some $a,a'\leq b$ would imply that $\omega$
    also contains an occurrence of $cc'$ for some $c,c'>b$, against
    \Cref{t:1sqlet}.
    
    We may then assume $|\delta_b|\leq |\omega|_b$, that is,
    $\sum_{c>b}|\omega|_c\leq\sum_{a\leq b}|\omega|_a$. By the above
    argument, this shows $cc'\notin L(\omega)$ for letters $c,c'>b$. As
    $\delta_b>0$, we also have
    \[\sum_{a<b}|\omega|_a<\sum_{c>b}|\omega|_c\leq
    \sum_{c\geq b}|\omega|_c,\]
    which shows $aa'\notin L(\omega)$ for $a,a'<b$. Hence, in order to prove
    statement~\ref{nodd}, we only need to show $ab,ba\notin L(\omega)$ for
    letters $a<b$. By contradiction, suppose $ab\in L(\omega)$, so that
    $cb\notin L(\omega)$ for any letter $c>b$, by \Cref{t:1sqlet}. Then every
    occurrence in $\omega$ of a letter $c>b$ would be followed by some $a<b$,
    so that $\sum_{c>b}|\omega|_c\leq\sum_{a<b}|\omega|_a$, contradicting
    $\delta_b>0$. Similarly, one obtains $ba\notin L(\omega)$.

    To complete the proof of statement~\ref{runs}, let us observe that every
    occurrence in $\omega$ of a letter $a<b$ is preceded by some letter $c>b$;
    moreover, every occurrence of $b$ is followed by either $b$ or some $c>b$,
    and every occurrence of a $c>b$ is followed by either $b$ or some $a<b$.
    Summing up, we obtain:
    \[\begin{split}
        |\omega|_b &=|\omega|_{bb}+\sum_{c>b}|\omega|_{cb},\\
        \sum_{c>b}|\omega|_c &=\sum_{c>b}|\omega|_{cb}
        +\sum_{a<b<c}|\omega|_{ca},\\
        \sum_{a<b}|\omega|_a &=\sum_{a<b<c}|\omega|_{ca},
    \end{split}\]
    which yields $\delta_b=|\omega|_b-|\omega|_{bb}$. In other words, $\omega$
    has exactly $\delta_b$ runs of $b$, as desired.
\end{proof}

As an immediate consequence, we obtain
\begin{lemma}
    Let $\omega\in\Ac$ be a cyclic singular word, and $b\in\A$ be such that
    $\delta_b\neq 0$. Then $bb$ occurs in $\omega$ if and only if
    $|\omega|_b>|\delta_b|$.
\end{lemma}

For $b\in\A$ and $x\in\A^*$, we let $\xi_b(x)$ be the word obtained from $x$ by
adding an occurrence of $b$ to all existing runs of consecutive $b$, as well as
in the middle of any occurrence of factors $dd'$ where $d,d'\in\A$ and either
$b<\min\{d,d'\}$ or $b>\max\{d,d'\}$.

\begin{example}
Let $\A=\{a<b<c<d\}$ and $x=abbcacad$. Then
\[\begin{split}\xi_a(x) &=aababacaacaad,\quad
\xi_b(x)=abbbcacad,\quad \xi_c(x)=acbcbccaccad,\quad\text{and}\\
\xi_d(x)&=adbdbdcdadcdadd.\end{split}\]
\end{example}

\begin{lemma}
\label{t:mono}
    For all $b\in\A$ and $x,x'\in\A^+$ such that $|x|=|x'|$,
    \[x\prec x'\iff\xi_b(x)\prec\xi_b(x').\]
\end{lemma}
\begin{proof}
Let $z\in\A^*$ and $d,d'\in\A$ be such that $x$ begins with $zd$ and $x'$
begins with $zd'$. If $z=\eps$, or if $z$ ends with $b$, then the assertion is
verified since $\xi_b(x)$ and $\xi_b(x')$ begin with $\xi_b(z)d$ and
$\xi_b(z)d'$, respectively. So let $r\neq b$ be the last letter of $z$.

Suppose $r<b$. If $b\leq d,d'$, again $\xi_b(x)$ and $\xi_b(x')$ begin with
$\xi_b(z)d$ and $\xi_b(z)d'$, so we are done. If $d<b$ (resp.~$d'<b$), then
$\xi_b(x)$ begins with $\xi_b(z)bd$ (resp.~$\xi_b(x')$ begins with
$\xi_b(z)bd'$). Hence in all cases we have $d<d'$ if and only if
$\xi_b(x)\prec\xi_b(x')$. The case $r>b$ is similar.
\end{proof}

In the following, we shall often use the simple fact that a word $X\in\A^+$ is
$\xi_b(x)$ for some $x$ (which is then uniquely determined by erasing one
letter from each run of consecutive $b$ in $X$) if and only if the following
conditions are all satisfied:
\begin{enumerate}
	\item $X$ does not contain any factor $dd'$ where the letters $d,d'$ are
	both larger than $b$ or both smaller than $b$;
	\item $X$ does not contain any factor $abc$ or $cba$, with $a,c\in\A$
	such that $a<b<c$;
	\item\label{notb} $X\neq b$, and $X$ does not begin with $bd$ or end with $db$, where
	$d\in\A\setminus\{b\}$.
\end{enumerate}

Let us observe that the definition of $\xi_b$ naturally applies to cyclic words
just as well. We still write $\xi_b$ for the resulting map on $\Ac$. Note,
however, that if $x\in\A^+$ represents $\omega\in\Ac$, and $d,d'$ are its first
and last letters, then $\xi_b(\omega)$ is represented by one of the following:
\[\begin{cases}
\xi_b(x)b &\text{ if }d,d'<b\,\text{ or }d,d'>b,\\
\xi_b(x)b^{-1} &\text{ if }d=d'=b,\\
\xi_b(x) &\text{ otherwise.}
\end{cases}\]
The above characterization of words in the image of $\xi_b$ also works in the
cyclic case, but with a simplified condition~\ref{notb}; one only needs to
check that the whole cyclic word is not $b$ (the remaining part would not make
sense for a cyclic word, after all).

We remark that on cyclic singular words, the map $\xi_b$ acts like the morphisms of
the free group over $\A$ considered by
M.~Lapointe~\cite[Chapter~IV]{Lap} to construct perfectly clustering words for the
Burrows–Wheeler transform.
\begin{lemma}
	\label{t:xising}
    Let $b\in\A$ and \(\omega\in\Ac\) be such that $\delta_b(\omega)\neq 0$. Then
    \(\omega\) is singular if and only if so is \(\xi_b(\omega)\).
\end{lemma}

\begin{proof}
    By definition, $\delta_b(\omega)=\delta_b(\xi_b(\omega))$, so that there
    is no ambiguity in writing $\delta_b$.
    Without loss of generality, let us then assume $\delta_b>0$.
    
    Suppose $\xi_b(\omega)$ is singular first. We need to show that every
    factorization $\omega=uv$ with $u,v\in\A^+$ is synchronizing. We may
    assume that $u$ begins and ends with distinct letters, since when
    $u=zu'\rev z$ for some $u',z\in\A^*$, the factorization $\omega=u\cdot v$
    is synchronizing if and only if so is $\omega=u'\cdot\rev zvz$. By
    symmetry, we may also assume $u\prec\rev u$. Let then $u=u_1\dotsm u_h$ and
    $v=v_1\dotsm v_k$, for $u_1,\dots,u_h,v_1,\dots,v_k\in\A$ such that
    $u_1<u_h$, and let us show that $v\preceq\rev v$ in all cases.
    
    In the following analysis, we leave out the cases where $v_1<b$ and
    $v_k\geq b$, or $v_1=b$ and $v_k>b$, since $v\preceq\rev v$ is then clear. On
    the contrary, we may rule out all cases where $v_1\geq b$ and $v_k<b$ (or
    $v_1>b$ and $v_k=b$) by finding a corresponding non-synchronizing
    factorization in $\xi_b(\omega)$, which is absurd.
    Since $\delta_b>0$, by \Cref{t:runs}, $\xi_b(\omega)$ contains no factor $aa'$
    where $a,a'\in\A,\; a,a'\leq b,\,$ and $aa'\neq bb$; by definition of $\xi_b$,
    the same must be true for $\omega$, and so we may also exclude all cases
    where $u_hv_1$ or $v_ku_1$ have such form.
%
    In the remaining cases, we obtain $v\preceq\rev v$ by \Cref{t:mono}, as some
    (necessarily synchronizing) factorization of $\xi_b(\omega)$ gives
    $\xi_b(v)\preceq\xi_b(\rev v)$. A detailed account of all possibilities
    follows; remember that $u_1<u_h$ is assumed.
    \begin{itemize}
        \item If
        $u_1,v_1,v_k>b$,
        then $\xi_b(\omega)=b\xi_b(u)b\cdot\xi_b(v)$ is synchronizing, and
        clearly $b\xi_b(u)b\prec\rev{b\xi_b(u)b}=b\xi_b(\rev u)b$, so that
        $\xi_b(v)\preceq\xi_b(\rev v)$.
        \item\label{aacc} If $u_h\leq b\;$ and
        $v_1,v_k>b$, then $\xi_b(\omega)=\xi_b(u)\cdot\xi_b(v)$ with
        $\xi_b(u)\prec\xi_b(\rev u)$. Therefore $\xi_b(v)\preceq\xi_b(\rev v)$.
        The same happens
        for
        $u_1>b\,$ and $v_1,v_k\leq b$.
        \item If $u_1\leq b,\; u_h,v_1>b,\,$ and $v_k>b,\,$
        then $\xi_b(\omega)=\xi_b(u)b\cdot\xi_b(v)$ with
        $\xi_b(u)b\prec b\xi_b(\rev u)$, so that $\xi_b(v)\preceq\xi_b(\rev v)$.
        \item If $u_1=v_1=v_k=b$,        
        then $\xi_b(\omega)=b^{-1}\xi_b(u)\cdot\xi_b(v)$ with
        $b^{-1}\xi_b(u)\prec\xi_b(\rev u)b^{-1}$, whence
        $\xi_b(v)\preceq\xi_b(\rev v)$.
        \item Assume $u_1=v_k=b,\,$ and $v_1>b$. This is
        impossible since $\xi_b(\omega)=b^{-1}\xi_b(u)\cdot b\xi_b(v)$ would
        not be synchronizing.        
        \item Finally, consider the case where
		$u_1>b,\; v_1>b,\,$ and $v_k\leq b$.
        This is impossible too, as
        $\xi_b(\omega)=\xi_b(u)\cdot b\xi_b(v)$ would not be synchronizing.
    \end{itemize}
	It is easy to check that all possibilities for $u_1,u_h,v_1$, and $v_k$ are
	covered; we may conclude that $\xi_b(\omega)$ being singular implies that
	$\omega$ is	singular.

    Conversely, suppose $\omega$ is singular. We need to prove that every
    factorization $\xi_b(\omega)=UV$ is synchronizing. As above, we may assume 
    $U=U_1\cdots U_h$ and $V=V_1\cdots V_k$ for some $h,k\geq 2$,
    $U_1,\ldots, U_h,V_1,\ldots, V_k\in\A$, and $U_1<U_h$. In the following
    case analysis, we skip all cases where $V_1<b$ and $V_k\geq b$ (or where
    $V_1\leq b$ and $V_k>b$), since that gives $V\preceq\rev V$ immediately. By
    \Cref{t:runs}, $\omega$ has no factor $aa'$ with
    $aa'\in\A,\; a,a'\leq b,\,$ and $aa'\neq bb$; by definition of $\xi_b$, the
    same is true for $\xi_b(\omega)$, so we also exclude all cases where
    $U_hV_1$ or $V_kU_1$ have such form.

    Let $U_1<b$ first, so that we may assume $V_k>b$.
    \begin{itemize}
    	\item If $U_h<b$, then
		$V_1>b$ for the same reason, and there exist
		words $u,v\in\A^*$ such that $\omega=uv$ and $U=\xi_b(u)$; by
		\Cref{t:mono}, this implies $u\prec\rev u$ and then $v\preceq\rev v$ as
		$\omega$ is singular.
		Since $V=\xi_b(v)$, we obtain $V\preceq\rev V$ by \Cref{t:mono}.
		%
		\item If $U_h=b$, we may assume $V_1\geq b$, and then $V_1>b$
		since otherwise $V\preceq\rev V$ is clear.
		Therefore $U_{h-1}\geq b,\,$ as
		$U_{h-1}U_hV_1$ occurs in $\xi_b(\omega)$; we have either $U=\xi_b(u)$
		if $U_{h-1}=b$, or $U=\xi_b(u)b$ otherwise, with $\omega=uv$ for some
		$u,v\in\A^+$. In both cases, we obtain $u\prec\rev u$, so that
		$v\preceq\rev v$.
		As $V_k>b$, we deduce $V=\xi_b(v)$, so that $V\preceq\rev V$ by
		\Cref{t:mono}.
		%
		\item If $U_h>b$, then $V_1\leq b$ since $U_hV_1$ occurs in
		$\xi_b(\omega)$. Hence $V\preceq\rev V$.
    \end{itemize}
    
    Now let $U_1\geq b$, whence $U_h>b$ and $V_1\leq b$. Suppose $U_1=b$
    first, so that we may assume $V_k\geq b$ and so the only case left to check
    is $V_1=V_k=b$.
        Thus we have $V_2\geq b$, and there exist $u,v$ such that
        $\omega=uv$ and $bU=\xi_b(u)$, with $u\prec\rev u$ and hence $v\preceq\rev v$.
        If $V_2=b$, then $V=\xi_b(v)$ or $V=\xi_b(v)b$, with $v$ beginning with
        $b$ and ending with $V_{k-1}$; thus $V_{k-1}\geq b$ and $V\preceq\rev V$.
        If $V_2>b$, we have either $V=b\xi_b(v)$ or $V=b\xi_b(v)b$, again with
        the same endings for $v$; therefore $V_{k-1}>b$ and $V\preceq\rev V$ by
        \Cref{t:mono}.
        
	Finally, let $U_1>b$, so that $V_k\leq b$. We have $U=\xi_b(u)$ and
	$\omega=uv$ for suitable $u,v$, with $u\prec\rev u$ and hence $v\preceq\rev v$.
	If $V_1,V_k<b$, then $V=\xi_b(v)$, so that $V\preceq\rev V$. Let then $V_1=b$,
	so that $V_2\geq b$. Then $V_k<b$ would imply $V=\xi_b(v)$ or $V=b\xi_b(v)$
	with $b$ beginning with $V_2\geq b$ and ending with $V_k<b$, a
	contradiction. Hence we may assume $V_k=V_1=b$. The same sub-case analysis
	for $V_2$ as in the case $U_1=b$ leads to the conclusion
	$V\preceq\rev V$.
\end{proof}

Let us observe that if $\xi_b(\omega)\in\Ac$ is
singular with Parikh vector $\pv$ and $\delta_{b}\neq 0$, then
\Cref{t:xising,t:runs} imply that $\omega$ is singular and has Parikh vector
$\pv-|\delta_b|\mathbf e_b$, where $\mathbf e_b$ is the unit vector indexed by
$b$ (that is, the Parikh vector of $b$).
Conversely, if $\omega$ is singular with $\delta_{b}\neq 0$, then applying
$\xi_{b}$ adds exactly $|\delta_{b}|$ occurrences of $b$ by
definition of $\xi_{b}$ and by \Cref{t:runs}, since the
resulting word $\xi_{b}(\omega)$ is singular by \Cref{t:xising}.
This suggests a partial recursive algorithm for constructing cyclic singular
words with a given Parikh vector $\pv=\pv^{(0)}=(n_d)_{d\in\A}$.
    \begin{alg}
    \label{a:non0}
    Input: a vector $\pv\in\nats^\A$.
        \begin{enumerate}
        \item Find the least $b\in\A$ such that $n_b\geq |\delta_b|$.
        \item If $\delta_b\neq 0$, restart from the vector
        $\pv^{(1)}=\pv-|\delta_b|\mathbf e_b$.
        \item Repeat the previous steps until reaching a vector $\pv^{(m)}$
        such that $\delta_{b_m}(\pv^{(m)})=0$ for the letter
        $b_m\in\A$ found at step 1.
        \item If $\pv^{(m)}$ is not a multiple of $\mathbf e_{b_m}$, exit
        (failure); otherwise let $\omega^{(m)}$ be the only cyclic word having
        Parikh vector $\pv^{(m)}$, i.e., $\omega^{(m)}=b_m^k$ with
        $k=|\omega^{(m)}|$.
        \item Go back through previous vectors using \Cref{t:xising} to obtain
        new words, that is: for $0\leq i<m$, if
        $\pv^{(i+1)}=\pv^{(i)}-|\delta_{b_i}|\mathbf e_{b_i}$,
        then set $\omega^{(b)}=\xi_{b_i}(\omega^{(i+1)})$.
    	\end{enumerate}
	\end{alg}
	
	\begin{theorem}
	\label{t:uniq}
	With the above notation and definitions, if \Cref{a:non0} succeeds then its
	output $\omega=\omega^{(0)}$ is the \emph{unique} cyclic singular word with
	Parikh vector $\pv$.
	\end{theorem}
	\begin{proof}
	Suppose the algorithm succeeds for $\pv$. Then the word $\omega$ is
	obtained from a constant cyclic word $\omega^{(m)}$ (which is trivially
	singular and unique) by applying $m$ many $\xi$ operations, so that
	$\omega$ is singular by \Cref{t:xising}.
	
	Now let $\widehat\omega\in\Ac$ be any cyclic singular word with Parikh
	vector $\pv$. By hypothesis, if $b$ is the letter found at step 1 for
	$\pv$, then $|\delta_b(\widehat\omega)|=|\delta_b(\pv)|>0$. As a
	consequence of \Cref{t:runs}, $\widehat\omega$ contains no factor $dd'$
	with $d,d'\in\A$ both larger or smaller than $b$, and no factor $abc$ or
	$cba$ with $a<b<c$.
	Since we can assume $\widehat\omega\neq b$, it follows that there exists
	$\widehat\omega^{(1)}\in\Ac$ such that
	$\widehat\omega=\xi_b(\widehat\omega^{(1)})$. By \Cref{t:runs},
	$\widehat\omega^{(1)}$ has Parikh vector $\pv^{(1)}$, so that by induction
	it coincides with $\omega^{(1)}$. It follows
	$\widehat\omega=\xi_b(\omega^{(1)})=\omega$.
	\end{proof}
		
	\begin{example}
	Let $\A=\{a<b<c<d\}$. The only cyclic singular word with Parikh vector
	$\pv=(3,3,4,2)$ is $\omega=acbcbcbcadad$. Indeed, as
	$3\geq\delta_b=|4+2-3|=3$,
	the next vector according to \Cref{a:non0} is
	$(3,0,4,2)$, which is then followed by
	\[(3,0,3,2),\,(3,0,2,2),\,(3,0,1,2),\,(0,0,1,2),\,(0,0,1,1),\,
	(0,0,1,0)=\pv^{(7)}\,.\]
	The corresponding cyclic singular words are: $\omega^{(7)}=c$, then
	\[cd,\,cdd,\,acadad,\,accadad,\,acccadad,\,accccadad,\,acbcbcbcadad
	=\omega\,.\]
	However, starting from $\mathbf u=(3,2,4,3)$ instead, we obtain no output.
	Indeed, the next vectors are $(3,2,2,3)$ and $(3,0,2,3)=\mathbf u^{(2)}$,
	but the procedure stops	here since the leftmost component greater or equal
	to the corresponding $\delta$ value is the second one, and
	$\delta_b(\mathbf u^{(2)})=0$. The cyclic words $accbccbdadad$ and
	$accbdaccbdad$, along with their reverses, are all singular with Parikh
	vector $\mathbf u$, but they cannot be obtained by repeated	application of
	$\xi$ operations starting from a constant word.
	\end{example}
	
	\begin{remark}
	Since the reverse of a (cyclic) singular word is still singular,
	\Cref{t:uniq} shows in particular that any cyclic word output by 
    \Cref{a:non0} is symmetric. This can also be proved more directly,
	from properties of $\xi$ maps. An easy consequence is that \Cref{a:non0}
    may only succeed for vectors having at most two odd entries.
	\end{remark}
	
	\begin{corollary}
	If $\C$ is a cyclic abelian class whose vector produces an output under
	\Cref{a:non0}, then $|\Sing(\C)|=1$. In particular, if
	$\A\subseteq \{2,3,4,\ldots,\}$, then $\C$ contains a unique cyclic word $\omega$ with the property that $\Ksc(\omega)=\max \{\Ksc(\nu) : \nu \in \C\}.$
	\end{corollary}

\end{document}